\documentclass[12pt]{amsart}
\usepackage{amssymb}
\usepackage{mathtools}
\usepackage[english,french]{babel}
\usepackage{epsfig}
\usepackage{mathptmx}
\usepackage{amsmath,amsfonts,amssymb}
\usepackage{mathtools}
\usepackage{mathrsfs}
\usepackage[all]{xy}
\usepackage{graphicx}
\usepackage{latexsym}
\usepackage{verbatim}
\usepackage{xcolor}
\newcommand{\hide}[1]{}
\numberwithin{equation}{section}

\def\MBP{\text{\rm MBP}}
\def\FBP{\text{\rm FBP}}

\newcommand{\D}{\mathbb D}

\newcommand{\R}{\mathbb R}

\newcommand{\C}{\mathbb C}

\newcommand{\B}{\mathcal B}

\newcommand{\Aut}{{\sf Aut}(\mathbb D)}

\def\Aut{{\sf Aut}}

\def\1#1{\overline{#1}}
\def\2#1{\widetilde{#1}}
\def\3#1{\widehat{#1}}
\def\4#1{\mathbb{#1}}
\def\5#1{\frak{#1}}
\def\6#1{{\mathcal{#1}}}

\newcommand{\mcite}[1]{\csname b@#1\endcsname}

\theoremstyle{theorem}

\def\Aut{{\sf Aut}}

\newtheorem{theorem}{Theorem}[section]

\theoremstyle{definition}

\theoremstyle{remark}
\newtheorem{remark}[theorem]{Remark}
\newtheorem{remarks}[theorem]{Remarks}

\newtheorem{conjecture}[theorem]{Conjecture}
\numberwithin{equation}{section}

\title[Critical sets]{The Nehari--Schwarz lemma and infinitesimal boundary\\[2mm] rigidity of bounded holomorphic functions}

\author[O. Roth]{Oliver Roth}
\address{O. Roth: Department of Mathematics, University of W\"urzburg, Emil Fischer Strasse 40, 97074, W\"urzburg, Germany.} \email{roth@mathematik.uni-wuerzburg.de}


\dedicatory{Dedicated to the  memory of Gabriela Kohr}

\long\def\REM#1{\relax}

\begin{document}
\selectlanguage{english}
\begin{abstract}
  We survey a number of recent generalizations and sharpenings of Nehari's  extension of Schwarz' lemma  for holomorphic self--maps of the unit disk. In particular, we discuss the case of infinitely many critical points and its relation to the zero sets and invariant subspaces for Bergman spaces, as well as the case of equality at the boundary.
\end{abstract}

\maketitle

\renewcommand{\thefootnote}{\fnsymbol{footnote}}
\setcounter{footnote}{2}
\setcounter{tocdepth}{1}
\tableofcontents

\section{Introduction} \label{sec:intro}

Let $\D=\{z \in \C \, : \, |z|<1\}$ be the open unit disk in the complex plane $\C$, and let
$\B$ denote the set of holomorphic functions from $\D$ into $\overline{\D}$. 
A \textit{finite Blaschke product of degree $n$}  is   a rational function $B \in \B$ of the form
\begin{equation} \label{eq:zeros}
B(z)=\eta \prod\limits_{j=1}^n \frac{z_j-z}{1-\overline{z_j} z} \, , \qquad |\eta|=1 \, ,
\end{equation}
with zeros $z_1, \ldots, z_n \in \D$, not necessarily pairwise distinct.  Hence the multiplicative building blocks of finite Blaschke products are exactly the elements of the group of  conformal automorphisms of $\D$, 
$$ \Aut(\D)=\left\{\eta \frac{z_0-z}{1-\overline{z_0} z} \, :\,  |\eta|=1, \, z_0 \in \D \right\}\, .$$

Blaschke products are omnipresent, and occur for instance  as fundamental normpreserving
factors in many important classes of holomorphic functions on $\D$. We refer
to the recent monograph \cite{GMR2018} and the references therein for an
state--of--the--art account of the properties and abundant applications of
\FBP, the set of all finite Blaschke products.
In this note, we discuss a number of recent generalizations of  Nehari's celebrated extension \cite{Neh1946} of Schwarz' lemma, a topic which is intrinsically related to FBPs, but which has not been treated in \cite{GMR2018}.

\medskip

As point of departure, we note that a  geometric--topological way of thinking about  (non--constant) {\FBP}s is to view
them as proper holomorphic self--maps of $\D$ or -- equivalently -- as finite
branched coverings of $\D$, see \cite[Chapter 3]{GMR2018}. From this view point, it  seems natural
to describe a finite Blaschke product $B$  not in terms of its zeros as in (\ref{eq:zeros}), but in terms of its  \textit{critical points}, that is, the zeros of its first derivative $B'$.
 That this is indeed possible is the content of
the following celebrated result of M.~Heins \cite{Heins1962} (see also
\cite[Chapter 6]{GMR2018},   \cite{Wal1950} and \cite{WP79, Bou1992,Z96,
  Ste2005, KR2008, KR2012, SW2019, Weg2020}).

\begin{theorem}[Heins 1962]\label{thm:H1}
  Let $c_1,\ldots, c_{n-1}$ be points in $\D$. Then there is a Blaschke product $B$ of degree $n$ with critical points $c_1, \ldots, c_{n-1}$ in $\D$ and no others. The Blaschke product $B$ is unique up to post--composition with an element of $\Aut(\D)$.
  \end{theorem}

The Blaschke product $B$ in Theorem \ref{thm:H1} can be  characterized as the essentially unique extremal function in a sharpened form of the Schwarz--Pick  inequality. This fundamental observation \cite{Neh1946} is due to Nehari in 1947.
 In order to state Nehari's result we need to introduce some notation.
 We denote by $\mathcal{C}_f$ the collection of all critical points of a non--constant function $f \in \B$ counting multiplicities.
 By slight abuse of language, we  call $\mathcal{C}_f$ the \textit{critical set of $f$} and write $\mathcal{C}_g \subseteq \mathcal{C}_f$ whenever each critical point of a function $g \in \B$ is also a critical point of $f \in \B$ of at least the same multiplicity. This is in accordance with standard practices, see \cite[\S 4.1]{DS}.

\begin{theorem}[The Nehari--Schwarz lemma] \label{thm:NS1}
  Let $f \in \B$ and let $B \in \FBP$ such that $\mathcal{C}_B \subseteq \mathcal{C}_f$. Then:
  \begin{itemize}
    \item[(i)] (Nehari--Schwarz inequality) 
  \begin{equation} \label{eq:NS1}
    \frac{|f'(z)|}{1-|f(z)|^2} \le \frac{|B'(z)|}{1-|B(z)|^2} \quad \text{ for all } z \in \D;
  \end{equation}
\item[(ii)] (Strong form of the Nehari--Schwarz lemma at an interior point)\\
  Equality holds in (\ref{eq:NS1})  for some point $z \in \D \setminus \mathcal{C}_B$ if and only if 
$f= T \circ B$ for some $T \in \Aut(\D)$.
\end{itemize}
\end{theorem}

%Property (ii) is called  the \textit{strong form of the Nahri--Schwarz lemma at the interior point $p \in \D \setminus \mathcal{C}_b$ }.

%\medskip

\begin{remark}[The Schwarz--Pick lemma]
In Theorem \ref{thm:NS1}, one can always take $B$
as a finite Blaschke product without  any critical points, that is, as  a
conformal automorphism of $\D$. In this  case, Theorem \ref{thm:NS1} reduces to the standard Schwarz--Pick lemma:
\begin{itemize}
  \item[(i)] (Schwarz--Pick inequality)
\begin{equation} \label{eq:sp}
  \frac{|f'(z)|}{1-|f(z)|^2} \le \frac{1}{1-|z|^2}  \quad \text{ for all } z \in \D;
  \end{equation}
    \item[(ii)]  (Strong form of the Schwarz--Pick  lemma at an interior point)\\
Equality holds in (\ref{eq:sp}) for some point $z \in \D$ if and only if $f \in \Aut(\D)$.
\end{itemize}
\end{remark}

The main purpose of  this note is to survey some recent sharpenings and extensions of the Nehari--Schwarz lemma. In Section \ref{sec:infinite} we discuss a generalization of the Nehari--Schwarz lemma which allows for 
taking into account \textit{infinitely many} critical points instead of only finitey many as in Theorem \ref{thm:NS1}. In Section \ref{sec:bergman} we describe  the connections with the specific Bergman space $A^2_1$, in particular its  zero sets and invariant subspaces. Our presentation is based on recent work of Kraus \cite{K2013}, Dyakonov \cite{Dyakonov2014, Dyakonov2015}, and Ivrii \cite{Ivrii2019, Ivrii2021}. In Section \ref{sec:boundary} we discuss the so--called strong form
of the Nehari--Schwarz lemma, that is, the case of equality in the Nehari--Schwarz inequality 
\textit{at the boundary} which has recently been obtained in \cite[Theorem 2.10]{BKR2020} as a special case of a general boundary rigidity theorem for conformal pseudometrics.
In order to make this paper self--contained we also provide a fairly concise proof of the Nehari--Schwarz inequality (\ref{eq:NS1}) in Section \ref{sec:proof}. The proof we give  is  slightly  different from the standard proofs which can  be found in  \cite[Corollary, p.~1037]{Neh1946} and \cite[Theorem 24.1]{Heins1962}.   
The Nehari--Schwarz lemma has found many further applications, for which we refer to other works such as  \cite{Bergweiler1998,GrahamMinda1999,LiuMinda1992,Minda1983,Ste2005}, for instance.

\section{Proof of the Nehari--Schwarz inequality} \label{sec:proof}

We give a  proof of the Nehari--Schwarz inequality (\ref{eq:NS1}) which is  based on the observation that a finite Blaschke product $B$ has the property that
\begin{equation} \label{eq:HeinsFBP}
\lim \limits_{|z| \to 1} \left( 1-|z|^2 \right) \frac{|B'(z)|}{1-|B(z)|^2}=1 \, .
\end{equation}
In fact, condition (\ref{eq:HeinsFBP}) \textit{characterizes} finite Blaschke products (Heins \cite{Heins1986}, see also \cite{KRR06} and \cite[Chapter 6.5]{GMR2018}). We point out that a simple and direct proof that (\ref{eq:HeinsFBP}) holds for any finite Blaschke product
$$ B(z)=\eta \prod \limits_{j=1}^n \frac{z-z_j}{1-\overline{z_j} z} $$
is possible by making appeal to an identity due to Frostman \cite{Fro}, namely
\begin{equation} \label{eq:Fro}
\frac{1-|B(z)|^2}{1-|z|^2}=\sum \limits_{k=1}^n \left( \prod \limits_{j=1}^{k-1} \left| \frac{z-z_j}{1-\overline{z_j} z} \right|^2 \right) \frac{1-|z_k|^2}{\left|1-\overline{z_k} z\right|^2} \, , \qquad |z| \not=1 \, ,
\end{equation}
and the elementary formula for the logarithmic derivative of $B$ given by 
\begin{equation} \label{eq:LogDer}
 \frac{B'(z)}{B(z)} =\sum \limits_{k=1}^n \frac{1-|z_k|^2}{\left(1-\overline{z_k} z \right) \left(z-z_k \right)} \, .
\end{equation}

Frostman's identity (\ref{eq:Fro}) can be easily established by induction, see \cite[p.~77]{GMR2018}. Clearly, (\ref{eq:Fro}) and (\ref{eq:LogDer}) immediately imply (\ref{eq:HeinsFBP}).

\medskip

Using (\ref{eq:HeinsFBP}) we now give a proof of Theorem \ref{thm:NS1} following very closely the standard proof of Ahlfors' lemma \cite{Ahlfors1938} with only minor modifications.

\begin{proof}[Proof of Theorem \ref{thm:NS1} (i)] Let $f \in \B$ be non--constant, so $\mathcal{C}_f$ is a discrete subset of $\D$.
  We consider the auxiliary function
  $$ u(z):=\log \left( \frac{|f'(z)|}{1-|f(z)|^2} \frac{1-|B(z)|^2}{|B'(z)|} \right) \, \, .$$
  Since $\mathcal{C}_B \subseteq \mathcal{C}_f$ (``including multiplicities''), we see that $u$ is well--defined and real analytic on $\D \setminus \mathcal{C}_f$. For each $\xi \in \mathcal{C}_f$ the limit
  $$ \lim \limits_{z \to \xi} u(z) \in \R \cup \{-\infty\}$$
  exists, so $u$ extends to  an upper semicontinuous function on $\D$ with values in $\R \cup \{-\infty\}$ which we continue to denote by $u$.
Now,  a straightforward computation reveals
  $$ \Delta u=4 \left( \frac{|B'(z)|}{1-|B(z)|^2} \right)^2 \left( e^{2 u}-1 \right) \, , \qquad z \in \D\setminus \mathcal{C}_f  \, .$$
  In particular,
  $$ u^+:=\max \{ u,0\}$$
  is subharmonic in $\D$. On the other hand, in view of the Schwarz--Pick inequality,
  $$ \left(1-|z|^2 \right) \frac{|f'(z)|}{1-|f(z)|^2} \le 1 \, , $$
  we deduce from (\ref{eq:HeinsFBP}) that
  $$ \limsup \limits_{|z| \to 1} u(z) \le 0 \, .$$
  Hence $u^+ \le 0$ in $\D$ by the maximum principle. This implies $u \le  0$ and completes the proof of (\ref{eq:NS1}). 
\end{proof}

\begin{remark}[Strong form of the Nehari--Schwarz lemma at an interior point]
The case of equality for the Nehari--Schwarz inequality (\ref{eq:NS1}) for some \textit{interior} point $z \in \D \setminus \mathcal{C}_B$  can be handled in a similar way as the case of equality at some interior point for Ahlfors' lemma, which has been treated in \cite{Heins1962,Royden1986,Minda1987,Chen,KRR06}. We refer to  e.g.~\cite[Remark 2.2 (d)]{KRR06} for the details.
  \end{remark}

\section{Infinitely many critical points} \label{sec:infinite}

We begin with an extension of the theorems of Heins' (Theorem \ref{thm:H1}) and Nehari--Schwarz (Theorem \ref{thm:NS1})  essentially due to Kraus \cite{K2013}.

\begin{theorem} \label{thm:kraus}
  Let  $\mathcal{C}$ be the critical set of a non--constant function in $\B$.
 Then there is a Blaschke product $B$  with critical set $\mathcal{C}$ such that
  $$  \frac{|f'(z)|}{1-|f(z)|^2} \le \frac{|B'(z)|}{1-|B(z)|^2} $$
  for all $z \in \D$ and any $f \in \B$ such that $\mathcal{C}_f \supseteq \mathcal{C}$. If equality holds at a single point $z \not \in \mathcal{C}$, then $f= T \circ B$ for some $T \in \Aut(\D)$.
  The Blaschke product $B$ is uniquely determined by $\mathcal{C}$  up to post--composition with an element of $\Aut(\D)$.
\end{theorem}  

See \cite{K2013}, while the case of equality has been settled in \cite{KR2013}.
The Blaschke product $B$ in Theorem \ref{thm:kraus} is called  \textit{maximal
  Blaschke product for $\mathcal{C}$}. The set of all maximal Blaschke
products will be denoted by \MBP.

\begin{remarks}[Properties of maximal Blaschke products]
\phantom{aa} \hspace*{1cm}
    \begin{itemize}
    \item[(a)] Maximal Blaschke products are indestructible: $f \in \MBP$, $T \in \Aut(\D)$ $\Longrightarrow$ $T \circ f \in \MBP$.
    \item[(b)] ($\FBP \subseteq \MBP$)\\
      Maximal Blaschke products for \textit{finite} sets $\mathcal{C}$ are \textit{finite} Blaschke products and vice versa, see \cite[Remark 1.2 (b)]{KR2013}.
      In particular, Theorem \ref{thm:kraus} generalizes Theorem \ref{thm:H1} and Theorem \ref{thm:NS1}.
    \item[(c)] Any maximal Blaschke product is uniquely determined by its critical set up to postcomposition with an element of $\Aut(\D)$. This does not hold for general infinite Blaschke products. Neat examples are the  nontrivial Frostman shifts
      $$ \pi_a(z):=\frac{a-\pi_0(z)}{1-\overline{a} \pi_0(z)} \, , \qquad a \in \D \setminus \{0\} \, ,$$
      of the standard singular inner function
      $$ \pi_0(z):=\exp \left(-\frac{1+z}{1-z} \right) $$
which are Blaschke products  without critical points.
\item[(d)] The accumulation points of the critical set of a maximal Blaschke product $B$  are exactly the accumulation points of its zero set,  and $B$ has an analytic continuation across any other point of the unit circle, see \cite[Theorem 1.4 and Corollary 1.5]{KR2013}.
  \item[(e)] The set of maximal Blaschke products is closed with respect to composition, see \cite[Theorem 1.7]{KR2013}.
\end{itemize}
  \end{remarks}

\section{Zeros sets and invariant subspaces for Bergman spaces} \label{sec:bergman}
  
  \begin{remark}[{\MBP}s and zero sets  in Bergman spaces] \label{rem:1}
    Theorem \ref{thm:kraus} shows in particular that a set
    $\mathcal{C}\subseteq \D$ is the critical set of a function in $\B$ if and
    only if it is the critical set of some maximal Blaschke product.
    It has been shown in \cite{K2013} that this is the case if and only if
    $\mathcal{C}$ is the zero set of a function in the Bergman space (\cite{DS,HKZ})
    $$ A_1^2=\left\{  \varphi :\D \to \C  \text{ holomorphic} \, : \, \iint_{\D} (1-|z|^2)\,
      |\varphi(z)|^2\, dxdy < \infty \right\}\, .$$
    Hence
    \begin{equation} \label{eq:MBP2}
      \text{MBP}/\Aut(\D)=\left\{\text{zero sets of } A^2_1\right\} \, .
    \end{equation}
    
    This can be seen as an analogue of the classical fact that up to a rotation ($=$ multiplication by a number $\eta \in \mathbb{S}^1=\{z \in \C \, : \, |z|=1\}$) the 
    zero sets of functions in the Hardy space $H^2$ are exactly the zero sets of Blaschke products, 
    $$ \text{\rm BP}/\mathbb{S}^1=\left\{\text{zero sets of } H^2\right\} \, .$$
    \end{remark}

    \begin{remark}[Critical sets in $\B$ and singly generated invariant
      subspaces in Bergman spaces] \label{rem:2}

      Remark \ref{rem:1} has a simple operator theoretic interpretation,
      cf.~\cite{DS,HKZ} for background.
A closed subspace of $A^2_1$ is called \textit{zero--based} if it is defined as the set of all
$A^2_1$--functions that vanish at a prescribed set of points in $\D$. Each
such subspace is \textit{invariant}, that is, invariant w.r.t.~to multiplication by $z$. Hence
(\ref{eq:MBP2}) can be trivially rewritten as
 $$  \text{MBP}/\Aut(\D)=\left\{\text{zero--based invariant subspaces of } A^2_1\right\} \, .$$
In particular, if we denote by $[H]$ the \textit{subspace generated by a
  function $H \in A^2_1$},
that is, the minimal closed invariant subspace of $A^2_1$ which contains $H$,
then each zero--based subspace of $A^2_1$ has the form $[B']$, meaning that
it is singly generated by the
derivative $B' \in A^2_1$ of  some   maximal Blaschke product $B$.   
    Combining this observation with the beautiful concept of asymptotic spectral synthesis of Nikol'skii \cite{Nik}
and 
    a deep result of Shimorin \cite{Shimorin2000} about approximation of  singly--generated
    invariant subspaces of Bergman spaces by zero--based subspaces, O.~Ivrii
    \cite{Ivrii2021}  has recently been led to the following striking conjecture
    $$  \text{Inner functions}/\Aut(\D)=\left\{\text{singly generated invariant
      subspaces of } A^2_1\right\} \, ,$$
    or in more explicit terms:
    \begin{conjecture}[Ivrii \cite{Ivrii2021}] \label{conj}
Any singly generated subspace of $A^2_1$ can be generated by the derivative of
an inner function. This inner function is uniquely determined up to postcomposition with a unit disk automorphism.
      \end{conjecture}

This conjecture can be seen as an analogue of the celebrated result of
Beurling that the invariant subspaces of $H^2$ are generated by inner
functions:
$$ \text{Inner functions}/\mathbb{S}^1=\left\{\text{invariant subspaces of } H^2\right\} \, .$$
      
    We refer to the original papers \cite{Ivrii2019, Ivrii2021} for details
    and a number of substantial results in support of Conjecture \ref{conj}.
    \end{remark}

\section{The strong form of the extended Nehari--Schwarz lemma at the boundary} \label{sec:boundary}
    
    We now return to Theorem \ref{thm:kraus} and discuss 
 the case of equality at the boundary.
  For this purpose it is convenient to denote by
 $$ f^h(z):=\left(1-|z|^2 \right) \frac{|f'(z)|}{1-|f(z)|^2}$$
  the \textit{hyperbolic derivative} of a holomorphic function $f : \D \to \D$, see \cite[Definition 5.1]{BeardonMinda2004}. If $ f\in \B$ and $B$ is a maximal Blaschke product with $\mathcal{C}_B \subseteq \mathcal{C}_f$, then
  $$ \frac{f^h}{B^h} : \D\setminus \mathcal{C}_B \to \R $$
  has a continuous extension to $\D$  which will still be denoted by $f^h/B^h$. Theorem \ref{thm:kraus} (see also \cite[Theorem 2.2 (b)]{KR2013}) implies:
  \begin{itemize}
    \item[(i)] (Extended Nehari--Schwarz inequality)
      $$ \frac{f^h(z)}{B^h(z)} \le 1 \quad \text{ for all } z \in \D; $$
      \item[(ii)] (Strong form of the extended  Nehari--Schwarz lemma at an interior point)
  $$ \frac{f^h(z)}{B^h(z)}=1 \quad \text{ for some point } z \in \D \quad \Longleftrightarrow \quad f=T \circ B\quad  \text{ for some  } T \in \Aut(\D)\,. $$
\end{itemize}
  Recently, a boundary version of this interior rigidity result for functions in $\B$
  has been obtained in \cite{BKR2020}:
 \begin{theorem}[The strong form of the generalized Nehari--Schwarz lemma at
   the boundary] \label{thm:BKR}
   %[Boundary rigidity with prescribed branching]
   \label{thm:branching}  Let  $\mathcal{C}$ be the critical set of a non--constant function in $\B$, 
$B$ a maximal Blaschke product  with critical set $\mathcal{C}_B=\mathcal{C}$, and $f \in \B$ such that $\mathcal{C}_f \supseteq \mathcal{C}$. 
If $$
\frac{f^h(z_n)}{B^h(z_n)}=1+o\left((1-|z_n|)^2\right)$$ 
for some sequence $(z_n)$ in $\D$ such that $|z_n| \to 1$, then $f=T \circ B$
for some $T \in \Aut(\D)$ and $f$ is a maximal Blaschke product.
\end{theorem}

The proof of Theorem \ref{thm:BKR} in \cite{BKR2020} is based on PDE methods,
in particular  a Harnack--type inequality for solutions of the Gauss curvature
equation, see \cite{BKR2020} for details. This approach also yields a  version
of the strong form of the Ahlfors--Schwarz lemma  \cite{Ahlfors1938, Chen, Heins1962,
  Minda1987, Royden1986, Yam1988} \textit{at the boundary}, see \cite[Theorem 2.6]{BKR2020}.
The special case $\mathcal{C}=\emptyset$ of Theorem \ref{thm:BKR} is the following boundary version of the strong form of the classical Schwarz--Pick lemma:

\begin{theorem}[The strong form of the Schwarz--Pick lemma at
   the boundary] \label{thm:BKR2} \label{thm:maindisk}
    Let $f: \D \to \D$ be holomorphic. If 
\vspace*{-0.0cm}   $$ %(1-|z|^2)  \frac{|f'(z)|}{1-|f(z)|^2}=
 f^h(z_n)=1+o((1-|z_n|)^2)$$
for some sequence $(z_n)$ in $\D$ such that $|z_n| \to 1$, then $f \in \Aut(\D)$.
 \end{theorem}

 The error term is sharp.
For $f(z)=z^2$ we have
$$ f^h(z)=
%\left(1-|z|^2 \right) \frac{|f'(z)|}{1-|f(z)|^2}=
\frac{2
  |z|}{1+|z|^2}=1-\frac{(1-|z|)^2}{1+|z|^2}=1-\frac{1}{2} \left(1-|z| \right)^2+o\left((1-|z|)^2\right) \qquad (|z| \to 1) \, .$$
Hence one cannot replace ``little $o$'' by ``big $O$'' in Theorem \ref{thm:BKR2}.
 Theorem \ref{thm:BKR2} can also be deduced from the inequality
  \begin{equation} \label{eq:golusin}
 f^h(z) \le \frac{f^h(0)+\displaystyle \frac{2
     |z|}{1+|z|^2}}{1+f^h(0)\displaystyle  \frac{2 |z|}{1+|z|^2}}   \quad
 \text{ for all } |z|<1\,  ,
 \end{equation}
 which has been proved by  Golusin
  (see \cite[Theorem 3]{Golusion1945} or \cite[p.~335]{Go}, and  independently by Yamashita
  \cite{Yamashita1994,Yamashita1997}, Beardon  \cite{Beardon1997}, and by Beardon \&
  Minda \cite{BeardonMinda2004,BeardonMinda2007} as part of their
elegant work on multi--point Schwarz--Pick lemmas. 
 With hindsight, inequality (\ref{eq:golusin}) is exactly the case $w=0$ in Corollary 3.7 of
\cite{BeardonMinda2004}.

\begin{remark}[The boundary Schwarz--Pick lemma and the boundary Schwarz lemma of Burns and Krantz]
From Theorem \ref{thm:maindisk}  one can easily deduce the well--known boundary Schwarz lemma of Burns and Krantz
\cite{BurnsKrantz1994}, which asserts that if  $f$ is a holomorphic selfmap of $\D$ such that
\begin{equation} \label{eq:burnskrantz1}
f(z)=z+o\left(|1-z|^3\right) \quad \text{ as } z \to 1\, ,
\end{equation}
then $f(z) \equiv z$. We refer to \cite[Remark 2.2]{BKR2020} for details.
\end{remark}

\begin{remark}
  Baracco, Zaitsev and Zampieri \cite{BaraccoZaitsevZampieri2006} have improved the boundary Schwarz lemma of Burns and Krantz by proving that if $f : \D \to \D$ is a holomorphic map such that 
  $$f(z_n) =z_n+o\left(|1-z_n|^3\right)$$
  for some sequence $(z_n)$ in $\D$ converging nontangentially to $1$, then $f(z) \equiv z$. Does the result of 
  Baracco, Zaitsev and Zampieri follow from Theorem \ref{thm:BKR2}?
\end{remark}

  We
  refer to \cite{Bol2008,Chelst2001,Dubinin2004, Osserman2000,
    Shoikhet2008,TauVla2001} and in particular to the survey \cite{EJLS2014} by
  Elin et al.~for more on boundary Schwarz--type lemmas in the  one variable setting.

\end{document}